\documentclass{article}[12pt]
\usepackage{amsfonts, amsmath, amssymb, amscd, amsthm, graphicx,enumerate,comment, enumitem}
\usepackage{hyperref}
\usepackage[usenames]{color}
\usepackage{tikz}
\usepackage{calc}
\usepackage{braket}
\usepackage{mathrsfs}
\usepackage[marginpar=3cm]{geometry}
\usepackage{overpic}
\usepackage{cleveref}
\crefname{enumi}{PV}{parts}
\usepackage{mathtools}
\usepackage{caption}
\usepackage{subcaption}
\usepackage{titlesec}

\usepackage{thmtools}
\usepackage{thm-restate}

\makeatletter
\def\blfootnote{\xdef\@thefnmark{}\@footnotetext}
\makeatother

\newcommand{\z}{\mathbb{Z}}
\newcommand{\N}{\mathbb{N}}
\newcommand{\acts}{\curvearrowright}
\newcommand{\AH}{\mathcal{AH}}
\newcommand{\AHG}{\mathcal{AH}(G)}
\newcommand{\Hll}{\mathcal{H}}
\newcommand{\HlG}{\mathcal{H}(G)}
\newcommand{\Ga}{\Gamma}

\newcommand{\R}{\mathbb{R}}
\newcommand{\e}{\varepsilon}
\newcommand{\nl}{\mathrm{(NL)}}
\newcommand{\wrp}{\operatorname{wr}}

\newcommand{\C}{\mathbb{C}}

\newcommand{\op}{\operatorname}
\newcommand{\G}{\Gamma}

\newtheorem{theorem}{Theorem}[section]
\newtheorem{lemma}[theorem]{Lemma}
\newtheorem{prop}[theorem]{Proposition}
\newtheorem{cor}[theorem]{Corollary}
\newtheorem{ques}[theorem]{Question}
\theoremstyle{definition}
\newtheorem{definition}[theorem]{Definition}
\newtheorem{example}[theorem]{Example}

\theoremstyle{remark}
\newtheorem{remark}[theorem]{Remark}
\numberwithin{equation}{section}

\begin{document}
\title{A survey on classifying hyperbolic actions of groups}

\author{Sahana H. Balasubramanya}

\date{}

\maketitle

\begin{abstract} This paper is a survey of results proved in recent years that pertain to classifying cobounded hyperbolic actions of any group $G$. In other words, we discuss results that allow us to describe the partially ordered set $\HlG$, first introduced in by Abbott-Balasubramanya-Osin. In certain cases, a complete classification of the poset is possible. In cases where this seems out of reach, we provide descriptions of large subsets of the poset. This covers a wide range of groups, including acylindrically hyperbolic groups, nilpotent groups, many solvable groups and ``Thompson-like" groups. We also cover the range of strategies utilized to obtain these classification results. As a result, we produce some new examples of groups with interesting $\HlG$ structure. \\

Subject class:  20F65, 20F16.
\end{abstract}

\tableofcontents


\section{Introduction}
\label{sec:intro}

Isometric actions on hyperbolic metric spaces, termed hyperbolic actions, have proven to be important tools in geometric group theory. They have been used to study the coarse geometry of a wide variety of groups, bounded cohomology (\cite{bf}), quotients of groups (\cite{dgo}), and decision problems for groups (\cite{geometry}). It is thus natural to consider the problem of classifying \emph{all} of the hyperbolic actions of a group. Indeed, this is a topic that has received considerable attention in recent years (see \cite{bestvinasurvey} for a recent survey). The problem also has links to the  Bieri-Neumann-Strebel invariants (see \cite{ABR}) and may thus also be useful for studying related notions such as property $R_\infty$ (see \cite{SSV}). 

 Although natural, requesting a complete classification of hyperbolic actions for any group $G$ is a highly complex problem to tackle. For instance, any group can be made to act trivially on any hyperbolic space. One can also use a combinatorial horoball construction by Groves-Manning \cite{GrovesManning} to construct many parabolic actions of a countable group $G$ on a hyperbolic space.  These actions, however, retain almost no information about the algebraic structure of the group, and so it is feasible to impose conditions on the action that exclude these actions from consideration.   

It turns out that it suffices to study \emph{cobounded} hyperbolic actions to solve this problem. Additionally, it makes sense to consider actions up to \emph{quasi-isometric} equivalence. i.e. spaces which share the same coarse geometry, and where the actions are, up to some bounded error, effectively the same. This is especially helpful as hyperbolicity is preserved under quasi-isometries. Together, these conditions rule out all trivial actions except on a point and all parabolic actions. Additionally, they allow us to introduce a partial order on the actions, and arrange them into a poset, which we call the poset of \emph{hyperbolic structures on $G$}, denoted $\HlG$. Roughly speaking, we consider one action to be smaller than another when the smaller action may be obtained by an operation of collapsing an equivariant family of subspaces in the larger action (see section \ref{sec:prelims} for details). Intuitively speaking, the larger action retains more information about the algebraic structure of the group; the action on a point (called the trivial action/structure) is thus always the smallest in this poset. 

However, even with these conditions, many groups classically studied in geometric group theory admit uncountably many distinct hyperbolic actions (see section \ref{sec:ahgrps} below), making the classification problem intractable even for well-understood groups. Although we can still describe large sections of the poset $\HlG$ (which has been useful for studying the groups nonetheless), it urges us to to attempt to solve this problem for groups with naturally fewer hyperbolic actions. The author made the first step towards this goal in \cite{Qp}, which was followed by the papers \cite{AR, ABR, ABR2}. These papers completely classify the poset $\HlG$ of certain classically studied solvable groups, which prompted the authors to together look into the classification problem for more general solvable groups. This work is covered in \ref{sec:solvgrps}. The overarching research goal of the work contained therein is to develop a theory that describes $\HlG$ for arbitrary metabelian, and eventually, arbitrary solvable groups. 

In another related direction, one can ask which groups, besides finite groups, admit a trivial $\HlG$ poset. i.e. which groups $G$ only admit the trivial hyperbolic structure? Such groups are natural examples of groups that are inaccessible for study from the perspective of hyperbolic actions. Nonetheless, it is interesting to study the common elements in the behavior and structure of such groups, especially those that obstruct the group from acting on a hyperbolic space in a meaningful way. The author, along with Fournier-Genevois-Sisto formalized the study of such groups in \cite{PropNL}. Their explorations cover many Thompson-like groups (covered in section \ref{sec:thompson}). In addition, other (amenable) groups are also studied and the paper provides criteria to check when these can admit hyperbolic actions without the necessity of explicitly constructing one. 

More recently, the paper \cite{bcfs} explores this classification problem in the setting of lattices in semi-simple groups. In the case that the simple factors have higher rank, the lattice was known to have no non-trivial hyperbolic actions by work from \cite{haettel}. However, the question remained unexplored in the case when the simple factors have rank one. It is natural to expect that the action of the factors will have connections to the actions of the lattices, as we can always create actions of the lattice by factoring through the rank one factors. \cite{bcfs} shows that, under mild conditions, these completely classify the actions of the lattice; providing a new source of groups for which $\HlG$ is completely understood; this work is covered in \ref{sec:lattices}.

This present paper was born from a desire to collect all these results into a common body of work. We hope such an overview will make it easier to reference and connect the different pieces of the larger puzzle for those interested in an overview or understanding this area of research, especially since it incorporates tools from different areas of mathematics such as commutative algebra, bounded cohomology, valuations and Lie groups.


\section{Definitions and Preliminaries}\label{sec:prelims}

In this section, we collect the necessary background information needed to understand the results in this paper. We formalize many definitions and theorems that will be used throughout the paper. We start with the precise description of $\HlG$ and its poset relation.

\subsection{Comparing generating sets and group actions} 
 
Throughout this paper, all group actions on hyperbolic spaces are assumed to be isometric, unless otherwise stated. Given a metric space $X$, we denote by $d_X$ the distance function on $X$. The $X$ will frequently be dropped from $d_X$ if it is understood by context. Moreover, if $G$ is a group and $S$ is a generating set of $G$, then we denote by $\|\cdot\|_S$ the word length on $G$ associated to $S$ and by $d_S$ the corresponding word metric $d_S(g,h)=\|gh^{-1}\|_S$. 

\begin{definition}\cite[Definition 1.1]{ABO}
Let $S$, $T$ be two (possibly infinite) generating sets of a group $G$. We say that $S$ is \emph{dominated} by $T$, written $S\preceq T$, if $$ \sup_{t\in T}\|t\|_S<\infty.$$  The relation $\preceq$ is a preorder on the set of generating sets of $G$, and therefore it induces an equivalence relation in the standard way:
$$
S\sim T \;\; \Leftrightarrow \;\; S\preceq T \; {\rm and}\; T\preceq S.
$$
This is equivalent to the condition that the Cayley graphs $\Ga(G,S)$ and $\Ga(G, T)$ with respect to $S$ and $T$ are $G$--equivariantly quasi-isometric. We denote by $[S]$ the equivalence class of a generating set $S$. The preorder $\preceq$ extends to a partial order on the equivalence classes in the standard way:
$$
[S]\preccurlyeq [T] \;\; \Leftrightarrow \;\; S\preceq T.
$$
\end{definition}

For example, all finite generating sets of a finitely generated group are equivalent and the equivalence class containing any finite generating set is the largest element. For every group $G$, the smallest element is $[G]$. Note also that this order is ``inclusion reversing": if $S$ and $T$ are generating sets of $G$ such that $S\subseteq T$, then $T\preceq S$.

To define a hyperbolic structure on a group, we first recall the definition of a hyperbolic space, for which we use the Rips condition. 

\begin{definition} A metric space $X$ is called \emph{$\delta$--hyperbolic} if it is geodesic and for any geodesic triangle $\Delta $ in $X$, each side of $\Delta $ is contained in the union of the closed $\delta$--neighborhoods of the other two sides.
\end{definition} 

\begin{definition}\cite[Definition 1.2]{ABO}
A \emph{hyperbolic structure} on $G$ is an equivalence class $[S]$ such that the Cayley graph $\Gamma (G,S)$ with respect to $S$ is hyperbolic. Since hyperbolicity of geodesic metric spaces is quasi-isometry invariant, this definition is independent of the choice of the representative $S$. We denote the set of hyperbolic structures by $\HlG$. It is a poset with the same partial order induced from the one on the equivalence classes.  \end{definition}

An equivalent description of the poset $\HlG$ can be made in terms of cobounded hyperbolic actions of $G$ up to coarsely equivariant quasi-isometry, which we now summarize.

\begin{definition}
The action $G\curvearrowright X$ is \emph{cobounded} if for some (equivalently any) $x\in X$ there exists $R>0$ such that every point of $X$ is distance at most $ R$ from some point of the orbit $Gx$.
\end{definition}

Given  a cobounded hyperbolic action $G\curvearrowright X$, there is an associated hyperbolic structure given by the Schwarz-Milnor Lemma. 

\begin{lemma}\cite[Lemma 3.11]{ABO}\label{lem:MS}
Let $G\curvearrowright X$ be a cobounded hyperbolic action of $G$. Let $B\subseteq X$ be a bounded subset such that $\displaystyle \bigcup_{g\in G} gB=X$. Let $D=\operatorname{diam}(B)$ and let $x\in B$. Then $G$ is generated by the set \[S=\{g\in G: d_X(x,gx)\leq 2D+1\},\] and $X$ is $G$--equivariantly quasi-isometric to $\Gamma(G,S)$.
\end{lemma}

 Thus, up to  equivariant quasi-isometries, hyperbolic actions of $G$ correspond to actions of $G$ on its hyperbolic Cayley graphs. There is an equivalence relation on hyperbolic actions that respects the poset structure, we direct the interested reader to \cite[Proposition 3.12]{ABO} for details. Thus, to classify the cobounded hyperbolic actions of $G$ is tantamount to classifying the elements of $\HlG$.

We will sometimes refer to representatives of hyperbolic structures, by which we mean the following : if $[S]\in \Hll(G)$, we say that $[S]$ is represented by a cobounded hyperbolic action $G\curvearrowright X$ if $G\curvearrowright X$ is equivalent to the action $G\curvearrowright \Gamma(G,S)$; that is, if there is a coarsely $G$-equivariant quasi-isometry $X\to \Gamma(G,S)$.

\subsection{Actions on hyperbolic spaces}

Let $G$ be a group acting on a hyperbolic space $X$, and denote the Gromov boundary of $X$ by $\partial X$. In general, $X$ is not assumed to be proper, and its boundary is defined as the set of equivalence classes of sequences convergent at infinity. We also denote by $\Lambda (G)$ the set of limit points of $G$ on $\partial X$. That is, $\Lambda (G) :=\partial X\cap \overline{Gx},$ where $\overline{Gx}$ denotes the closure of a $G$--orbit in $X\cup \partial X$, for any choice of basepoint $x\in X$.  This definition is independent of the choice of $x\in X$. 

Recall that a loxodromic element $g$ in a hyperbolic action on $X$ is an infinite order element such that the map $n \to g^nx$ is a quasi-isometric embedding for some (equivalently any) $x \in X$. Every loxodromic element $g \in G$ has exactly 2 limit points $g^{\pm \infty}$ on the Gromov
boundary $\partial X$. Loxodromic elements $g, h \in G$ are called independent if the sets $\{ g^{\pm\infty}\}$ and $\{ h^{\pm \infty}\}$ are disjoint. 

The following theorem summarizes the classification of group actions on hyperbolic spaces due to Gromov \cite[Section 8.2]{Gromov} and the some related results from \cite[Propositions 3.1 and 3.2]{Amen}.

\begin{theorem}\label{ClassHypAct}
Let $G$ be a group acting on a hyperbolic space $X$. Then exactly one of the following conditions holds.
\begin{enumerate}
\item[1)] The action of $G$ is \emph{elliptic}. i.e.$|\Lambda (G)|=0$. Equivalently,  $G$ has bounded orbits. 

\item[2)] The action of $G$ is parabolic. i.e.$|\Lambda (G)|=1$. A parabolic action cannot be cobounded. 

\item[3)] The action of $G$ is lineal. i.e.$|\Lambda (G)|=2$. Equivalently, $G$ contains a loxodromic element and any two loxodromic elements have the same limit points on $\partial X$. A lineal action $G \acts X$ is said to be \emph{orientable} if $G$ fixes its limit points on $\partial X$ \emph{pointwise}, and \emph{non-orientable} otherwise. 

\item[4)] The action of $G$ is non-elementary. i.e.$|\Lambda (G)|=\infty$. Then $G$ always contains loxodromic elements. In turn, this case breaks into two subcases.
\begin{enumerate}
\item[(a)] The action is \emph{quasi-parabolic}. i.e.$G$ fixes a point of $\partial X$. Equivalently, any two loxodromic elements of $G$ have a common limit point on the boundary. 
\item[(b)] The action of $G$ is of \emph{general type}. i.e. $G$ does not fix any point of $\partial X$. 
\end{enumerate}
\end{enumerate}
\end{theorem}

It follows immediately from the above that for any group $G$, $$\Hll(G)=\Hll_e(G)\sqcup \Hll_{\ell} (G)\sqcup \Hll_{qp} (G)\sqcup \Hll_{gt}(G)$$
where the sets of elliptic, lineal, quasi-parabolic, and general type hyperbolic structures on $G$ are denoted by $\Hll_e(G)$, $\Hll_{\ell} (G)$, $\Hll_{qp} (G)$, and $\Hll_{gt}(G)$ respectively. Namely, $[S]\in \mathcal H(G)$ lies in $\mathcal H_\ell(G)$ if the action $G\curvearrowright \Gamma(G,S)$ is lineal, and similar definitions hold for the posets $\Hll_e,\Hll_{qp},$ and $\Hll_{gt}$. Note that $\Hll_e(G) = \{ [G]\}$ always and this is always the smallest structure in the poset. We direct the reader to \cite[Section 4]{ABO} and \cite[Theorem 4.6]{ABO} for further explanation. 

\subsection{The Busemann pseudocharacter.} 
\label{sec:busemann}
A function $q\colon G\to \mathbb R$ is a \emph{quasi-character} (or \emph{quasi-morphism}) if there exists a constant $D$ such that $$|q(gh)-q(g)-q(h)|\le D$$ for all $g,h\in G$. We say that $q$ has \emph{defect at most $D$}. If, in addition, the restriction of $q$ to every cyclic subgroup of $G$ is a homomorphism, then $q$ is called a \emph{pseudocharacter} (or \emph{homogeneous quasi-morphism}). 

Given any action of a group $G$ on a hyperbolic space $X$ fixing a point on $\partial X$, one can associate a natural pseudocharacter $\beta$ called the \emph{Busemann pseudocharacter}. If $\beta$ is a homomorphism, then the action $G\curvearrowright X$ is called \emph{regular}. As we do not require the exact definition of $\beta$, we refer the reader to \cite[Sec. 7.5.D]{Gromov} and \cite[Sec. 4.1]{Man} for details. An element $g\in G$ is loxodromic with respect to the action of $G$ on $X$ if and only if $\beta(g)\ne 0$.  In particular, $\beta$ is not identically zero whenever $G \curvearrowright X$ is quasi-parabolic or orientable lineal. In the reverse direction, given a pseudocharacter on a group $G$, one can always construct an orientable lineal action, as evidenced by 

\begin{lemma}[{\cite[Lemma 4.15]{ABO}}]\label{lem:constolineal} Let $p\colon G \to \R$ be a non-zero pseudocharacter. Let $C$ be any constant such that the defect of $p$ is at most $C/2$ and there exists a value of $p$ in the interval $(0,C/2)$. Let $$ X =X_{p,C} =\{g \in G: |p(g)| < C\}.$$ Then $X$ generates $G$ and the map $p \colon (G,d_X) \to \R$ is a quasi-isometry. In particular, $[X]$ is orientable lineal. \end{lemma}

\subsection{Acylindrical actions}

There are many groups arising from geometric and topological settings that admit natural actions on hyperbolic spaces; yet are not hyperbolic or relatively hyperbolic groups. Nonetheless, they do share many interesting properties with hyperbolic groups and this initiated the study of non-geometric actions, including WPD actions. In a continued vein, acylindrical actions were first defined by Sela for groups acting on trees \cite{Sela} and later generalized by Bowditch for general metric spaces \cite{Bowditchacyl}. 

\begin{definition} An action of a group $G$ on a metric space $X$ is called acylindrical if for every $\e >0$, there exist constants $R, N \geq 0$ (depending only on $\e$) such that for any $x, y \in X$ satisfying $d(x,y) \geq R$, we have that $$ \# \{ g \in G \mid d(x, gx) \leq \e \text{ and } d(y, gy) \leq \e \} \leq N.$$ 
\end{definition}

Acylindricity can be thought of as a kind of properness of the action on $X \times X$ minus a ``thick diagonal". However, acylindricity is not a true generalization of the notion of a proper action. For instance, one can use the combinatorial horoball construction \cite{GrovesManning} to construct proper actions that are not acylindrical.  

In contrast to general hyperbolic actions, acylindrical actions are much more restrictive, as evidenced by the following theorem of Osin \cite[Theorem 1.1]{Osin}. 

\begin{theorem}\label{thm:acylclass} Let $G$ be a group acting acylindrically on a hyperbolic space. Then $G$ satisfies exactly one of the following three conditions.
\begin{enumerate}
    \item[(a)] $G$ has bounded orbits.
    \item[(b)] $G$ is virtually cyclic and contains a loxodromic element.
    \item[(c)] $G$ contains infinitely many independent loxodromic elements.
\end{enumerate}
\end{theorem}

In other words, acylindrical actions can be either elliptic, lineal or of general type. The following may then be taken as a definition -- a group $G$ is acylindrically hyperbolic if it is not virtually cyclic and admits an acylindrical action on a hyperbolic space with unbounded orbits. For such groups, one can ask if it is possible to achiev a complete classification of acylindrical actions on hyperbolic spaces. This allows us to define the corresponding subposet $\AHG$ of $\HlG$ as $$\AHG = \{ [X] \in \HlG \mid G \acts \Ga(G,X) \text{ is acylindrical} \}.$$

In the next section, we begin the survey of results that study the structure of $\AHG$, and how this influenced the direction of subsequent research on solvable groups.


\section{Acylindrically hyperbolic groups}\label{sec:ahgrps}

It follows easily from Theorem \ref{thm:acylclass} that if the group is not acylindrically hyperbolic, then either $\AHG = \{[G]\}$ or $G$ is virtually cyclic and $\AHG = \{ [G], [X]\}$, where $X$ is a finite generating set for $G$. Similarly, it follows that for an acylindrically hyperbolic group $G$, $$\AHG = \AH_e(G) \sqcup \AH_{gt}(G) = \{[G]\} \sqcup \AH_{gt}(G),$$ and so the classification problem in this case reduces to studying acylindrical general type actions of the group. In practice, however, this might be very hard to do due to the large cardinality of this poset. 

\begin{theorem}\cite[Theorem 2.6]{ABO}\label{thm:ahisbig} For an acylindrically hyperbolic group, $|\AHG| \geq 2^{\aleph_0}$. Further, $\AHG$ contains a copy of $P(\omega)$; the poset of all subsets of $\N$, ordered by inclusion.
\end{theorem}

The proof of this statement makes use of hyperbolic structures on groups induced from hyperbolic structures on subgroups. Specifically, one can show that if $H$ is a hyperbolically embedded subgroup of $G$, then $\AH(H)$ embeds into $\AHG$. Hyperbolically embedded subgroups are a generalization of peripheral subgroups of relatively hyperbolic groups; we omit going to into details here and refer the curious reader to \cite[Section 5]{ABO}. However, the result is applied to prove that $\AHG$ is sufficiently complicated as follows: every acylindrically hyperbolic group $G$ contains a hyperbolically embedded subgroup isomorphic to $F_2\times K$, where $F_2$ is free of rank $2$ and $K$ is a finite group \cite[Theorem 2.24]{dgo}. It then suffices to show that $\mathcal {AH}(F_2\times K)$ is sufficiently complicated, and the latter poset is much easier to understand. To this end, we use small cancellation techniques to construct many general type actions for any non-elementary hyperbolic group \cite[Theorem 2.11]{ABO}, thereby also generalizing a former result of Ivanov.

The statement of \cite[Theorem 2.11]{ABO} is infact much stronger. Indeed, it additionally shows that hyperbolic structures are not determined by the set of their loxodromic elements, even though  equivalent actions have the same sets of loxodromic elements. The theorem shows that for every non-elementary $[X] \in \AHG$, there are uncountably many inequivalent acylindrically hyperbolic structures with the same set of loxodromic elements. 

Besides the classification problem, another motivation for studying $\AHG$ comes from the desire to understand \emph{universal} actions. An element $g \in G$ is called a \emph{generalized loxodromic} if $g$ acts as a loxodromic element in \emph{some} acylindrical action of $G$ on a hyperbolic space. An acylindrical action of $G$ on a hyperbolic space $X$ is said to be universal if every generalized loxodromic is a loxodromic for $G \acts X$. Not every acylindrically hyperbolic group admits a universal action (eg: Dunwoody's group \cite{dunwoody}). In the context of $\AHG$, one can consider the related question of when the poset is \emph{accessible}.

\begin{definition}
We say that a group $G$ is \emph{$\AH$-accessible} if $\AHG$ contains the largest element.
\end{definition}

If a largest acylindrical action exists, it is necessarily universal. While there exist $\AH$-inaccessible groups -- even finitely generated or finitely presented ones (see \cite[Section 7.1]{ABO}) -- many groups traditionally studied in the geometric group theory are $\AH$-accessible.

\begin{theorem}\cite[Theorem 2.18]{ABO} The following groups are $\AH$-accessible.
\begin{enumerate}
\item[(a)] Finitely generated relatively hyperbolic groups whose parabolic subgroups are not acylindrically hyperbolic.
\item[(b)] Mapping class groups of punctured closed surfaces. The largest action is the action on the associated Curve complex.
\item[(c)] Right-angled Artin groups. The largest action is the action on the extension graph. 
\item[(d)] Fundamental groups of compact orientable $3$-manifolds with empty or toroidal boundary. The largest action is either on a point or on the Bass-Serre tree associated to the JSJ decomposition of the space. 
\end{enumerate}
\end{theorem}

In fact, the work in \cite{ABO} actually proves something stronger. It shows that the groups listed in the above theorem are \emph{strongly} $\AH$-accessible. i.e. they contain a largest acylindrical action that dominates all other (not necessarily cobounded) acylindrical actions, with the actions listed above being the largest in this weaker setting. 

Following an early draft of \cite{ABO}, parts (b), (c), and the special case of part (d) when the 3-manifold has no Nil or Sol in its prime decomposition of the above theorem are also proven in \cite{ABD}. \cite{ABD} also adds to the list of examples of $\AH-$accessible groups using different methods, as summarized in the following theorem. 

\begin{theorem}\cite[Theorem A]{ABD} Every hierarchically hyperbolic group admits a largest acylindrical action. 

In particular, this includes groups that act properly and cocompactly on a special CAT(0) cube complex, and more
generally any cubical group which admits a factor system. This includes right angled Artin groups, right-angled Coxeter groups, and many examples from \cite{HagenSusse}.
\end{theorem}

The result about right angled Artin groups has also been generalized in another direction to graph products of groups. 

\begin{theorem}\cite[Corollary D]{MV} 
    Let $\Ga$ be a finite simplicial graph and let $\mathcal{G} = \{ G_v \mid v \in V(\Ga) \}$ be
a collection of infinite groups. Suppose that for each isolated vertex $v \in V(\Ga)$,
the group $G_v$ is strongly $\AH$-accessible. Then the associated graph product of group $\Ga\mathcal{G}$ is strongly $\AH$-accessible.
Furthermore, suppose that  \begin{enumerate}
\item $\Ga$ has no isolated vertices,
\item $X$ is
the Cayley graph of $\Ga \mathcal{G}$ with respect to the generating set $S = \{ G_v \mid v \in V(\Ga)\}$, 
\item $\mathcal{C}X$ is the contact graph associated to $X$.
\end{enumerate}

Then the action $\Ga \mathcal{G} \acts \mathcal{C}X$,is the largest acylindrical action of $\Ga \mathcal{G}$ on a hyperbolic metric
space.
\end{theorem}

The above theorem relies on work from \cite{Gen}, which shows that the space $X$ in the statement above is a quasi-median space, allowing for a well defined contact graph $\mathcal{C}X$ to be constructed in \cite{MV}.

Recently, yet another class of groups was proven to have a largest acylindrical action. In the following theorem, the term \emph{admissible graph of groups} refers to a class of groups introduced by Croke and Kleiner to abstract the structure of $\pi_1 M$, when $M$ is a non-geometric graph manifold \cite{crokekleiner}. Roughly speaking, an admissible graph of groups is a nontrivial finite graph of groups $\mathcal G$ where each edge group is $\z^2$ and each vertex group $G_\mu$ has infinite cyclic center $Z_\mu$ with quotient a non-elementary hyperbolic group. Additionally, the various edge groups need to be pairwise non-commensurable inside each vertex group.

\begin{theorem}\cite[Corollary 5]{HRSS} Let $\mathcal{G}$ be an admissible graph of groups, and let $G = \pi_1\mathcal{G}$. Then $G$ is $\AH-accessible$ and the largest action is the action on the associated Bass–Serre tree.
\end{theorem}


\section{Solvable groups}\label{sec:solvgrps}

As seen from the results stated in the previous section, the class of acylindrically hyperbolic groups has a ``wild" behaviour, which makes obtaining a complete description of $\HlG$ (even $\AHG$) too intractable. Further, the proof of Theorem \ref{thm:ahisbig} suggests that free subgroups of rank $\geq 2$ could be a source of similar issues in classifying hyperbolic actions of any group that admits a non-elementary hyperbolic action. 

It is therefore, natural, to restrict our focus to the smaller class of groups where such issues do not arise -- this makes the class of solvable groups a natural candidate to consider. Indeed, solvable actions contain no non-elementary free subgroups, and thus admit non general type actions.

Another motivation to focus on solvable groups comes from the following fact: For every $n \in \N$, the authors of \cite{ABO} succeeded in producing examples of groups $H_n$ such that $|\Hll_{\ell}(H_n)| = n$; as well as examples of groups $G_n$ such that $|\Hll_{gt}(G_n)| = n$ (see \cite[Theorems 4.22 and 4.28]{ABO}). However, producing examples of groups with finitely many quasi-parabolic structures remained elusive. As solvable groups can only admit quasi-parabolic non-elementary actions, studying the structure of $\HlG$ for a solvable group $G$ could provide an indication to a source of examples in this vein.

\subsection{Abelian-by-cyclic and metabelian groups}

The first step in this direction was taken by the author in \cite{Qp}, which describes a part of the poset for groups of the form $G =H \wrp \z$; see \cite[Theorem 1.5]{Qp}. In particular, the work from \cite{Qp} describes the poset $\HlG$ completely for the Lamplighter groups $\mathcal{L}_n = (\z/ n\z) \wrp \z$ for $n \geq 2$, and shows that $|\Hll_{qp}(\mathcal{L}_n)|$ is twice the number of proper divisors of $n$. To the best of our knowledge, these were the first examples of groups with finitely many quasi-parabolic structures produced in the literature.

The main tool used in this work was the machinery of confining subsets from \cite{Amen}, which allows us to classify quasi-parabolic actions of $G$ in terms of \emph{confining} subsets of $H$; see \cite[Section 4 and Proposition 4.6]{Amen}. This machinery applies best in the case when the group has abelianization of rank 1, and was also used in \cite{AR,Largest} to describe $\HlG$ in the cases for solvable Baumslag Solitar groups $BS(1,n) = \z\left[ \frac{1}{n} \right] \rtimes \z$ and fundamental groups of the Anosov mapping torus $\z^2 \rtimes_\phi \z$, for $\phi \in SL(2, \z)$; see \cite[Theorem 1.1]{AR} and \cite[Theorem 1.3]{Largest}. Roughly speaking, if $G = H \rtimes \z$, then $Q \subset H$ is confining if elements of $H$ are \emph{attracted} into $Q$ under the action of $\z$, and $Q$ should be \emph{almost closed} under the group operation. For the precise definition, we refer the reader to \cite[Section 4]{Amen}. Aside from the inherent interest of classifying hyperbolic actions, the problem is also connected to the computation of Bieri-Neumann-Strebel (BNS) invariants of finitely generated groups \cite{brown}.

The paper \cite{ABR2} followed from the realization that the classification results contained in \cite{Qp, AR, Largest} fit within a common algebraic framework. By identifying this common framework and developing tools for working with groups that fit within it, the authors were able to classify the hyperbolic actions of numerous abelian-by-cyclic groups, and recover the key results from \cite{Qp, AR}. The main idea introduced in \cite{ABR2} was a method to convert the classification problem into a commutative algebra problem. The work utilizes confining subsets, radix representations, completions, and valuations on commutative rings, thereby surprisingly building strong relations between hyperbolic actions and commutative geometry. Before stating the main results, we describe the common framework in which they apply.

Consider a ring $R$ that is generated as a $\z$--algebra by an element $\gamma$. Assuming that $\gamma$ is neither a unit nor a zero divisor, we may define an ascending HNN extension $G\colon= G(R,\gamma)$ of $R$ by the endomorphism of (the abelian group) $R$ defined by multiplication by $\gamma$.  Assuming some natural restrictions on the algebra of $R$, one may write elements of $R$ as \emph{radix representations} of the form $a_0+a_1\gamma+a_2\gamma^2+\cdots$ and ensure that these are somewhat well-behaved. Whenever this setup holds, we can derive a description of the poset $\mathcal H(G)$ in terms of the \emph{$(\gamma)$-adic completion $\widehat R$ of $R$}. 

\begin{theorem}\label{thm:main}\cite[Theorem 1.1]{ABR2}
 Under the assumptions on $G$ as above, the poset $\mathcal H(G)$ consists of the following structures (see Figure \ref{fig:metabelian}): a single elliptic structure, which is dominated by a single lineal structure, and two subposets $\mathcal P_-(G)$ and $\mathcal P_+(G)$ that intersect in the single lineal structure. All other structures from $\mathcal P_+(G)$ and $\mathcal P_-(G)$ are quasi-parabolic.  The subposet $\mathcal P_+(G)$ is isomorphic to the poset of ideals of the $(\gamma)$-adic completion $\widehat R$ considered up to multiplication by $\gamma$. The poset $\mathcal P_-(G)$ is a lattice. Moreover, each element of $\mathcal P_+(G)$ can be represented by a quasi-parabolic action on a simplicial tree.
\end{theorem}

The proof of Theorem \ref{thm:main} first associates to each hyperbolic action a confining subset (unique up to multiplication by $\gamma$), and then describes the confining subsets in terms of ideals of the completion; this strategy is used for other results contained in the paper as well. The theorem broadly expands our understanding of $\mathcal H(G)$ for a large class of groups $G$.  In particular, roughly half the work of describing $\mathcal H(G)$ is reduced to the algebraic problem of classifying the ideals of the completion $\widehat R$ up to equivalence of ideals under multiplication by $\gamma$. In general though, the lattice $\mathcal P_-(G)$ described in Theorem \ref{thm:main} can be hard to get a handle on. However, the work in \cite{ABR2} also completely describes $\mathcal P_-(G)$ for a wide class of metabelian groups $G$, under some additional assumptions.

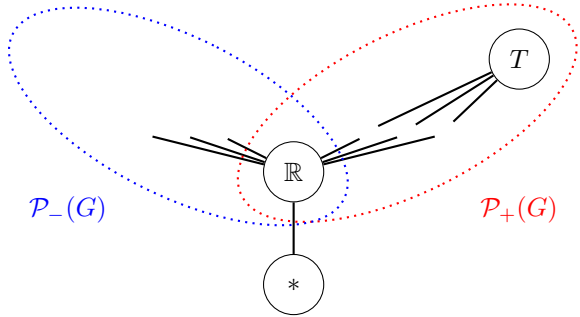
\begin{figure}[ht]
\centering

\begin{tikzpicture}[scale=0.5]

\node[circle, draw, minimum size=0.8cm] (triv) at (0,-3) {$*$};
\node[circle, draw, minimum size=0.8cm] (lin) at (0,0) {$\R$};
\node[circle, draw, minimum size=0.8cm] (bs) at (6,3) {$T$};

\node (low1) at (2,1) {};
\node (low2) at (3,1) {};
\node (low3) at (4,1) {};

\node (up1) at (2,1.1) {};
\node (up2) at (3,1.1) {};
\node (up3) at (4,1.1) {};

\node (Low1) at (-2,1) {};
\node (Low2) at (-3,1) {};
\node (Low3) at (-4,1) {};

\draw[thick] (triv) -- (lin);
\draw[thick] (lin) -- (low1);
\draw[thick] (lin) -- (low2);
\draw[thick] (lin) -- (low3);

\draw[thick] (lin) -- (Low1);
\draw[thick] (lin) -- (Low2);
\draw[thick] (lin) -- (Low3);

\draw[thick] (up1) -- (bs);
\draw[thick] (up2) -- (bs);
\draw[thick] (up3) -- (bs);

\draw[thick, dotted, red, rotate=296.5] (-0.05, 3.4) ellipse (60pt and 140pt);

\draw[thick, dotted, blue, rotate=63.5] (-0.05, 3.4) ellipse (60pt and 140pt);

\node[red] (ideals) at (6,-1) {$\mathcal{P}_+(G)$};

\node[blue] (others) at (-6,-1) {$\mathcal{P}_-(G)$};

\end{tikzpicture}
\caption{$\mathcal{H}(G)$}
\label{fig:metabelian}
\end{figure}

An \emph{abelian-by-cyclic} group is one of the form $G=A\rtimes \z$ where $A$ is abelian. We focus on torsion-free, finitely presented abelian-by-cyclic groups. An abelian-by-cyclic group $G$ can be described as an ascending HNN extension of a free abelian group $\z^n$ \cite{finpres}. That is, there is an endomorphism of $\z^n$, represented by a matrix $\gamma\in M_n(\z)$ with $\det \gamma\neq 0$, such that:
\[
G=\left\langle \z^n,t: tzt^{-1}=\gamma z,\, \forall z\in \z^n\right\rangle.
\]

In this case, results from \cite{ABR2} completely describe $\mathcal H(G)$ under two hypotheses on the matrix $\gamma$; namely that $\gamma$ is expanding (i.e. all eigenvalues of $\gamma$ lie outside the unit disk in $\C$) and that $\z^n$ is a cyclic $\z[x]$-module, where the action of $x$ is induced by that of $\gamma$. We call a matrix $\gamma$ with these  properties \emph{admissible}. 

We consider the prime factorization of the minimal polynomial (which, in this case, is also the characteristic polynomial \cite[Section 7.1]{hoffman_kunze}) $p=up_1^{n_1}\cdots p_r^{n_r}$ in the formal power series ring $\z[[x]]$, where $u\in \z[[x]]$ is a unit and each $p_i\in\z[[x]]$ is an irreducible power series.

\begin{theorem}\label{thm:char=min}\cite[Theorem 1.2]{ABR2}
Let $G$ be an ascending HNN extension of $\z^n$ by an admissible matrix $\gamma$. Then $\mathcal H(G)$ is as described in Theorem \ref{thm:main}. Further, 
\begin{enumerate}
\item $\mathcal P_+(G)$ is isomorphic to to the poset of divisors of $p=up_1^{n_1}\cdots p_r^{n_r}$ considered up to multiplication by a unit;
\item $\mathcal P_-(G)$ is isomorphic to the poset of  subspaces  of $\R^n$ that are invariant under $\gamma$; and
\item each element of $\mathcal P_-(G)$ is represented by an action on a quasi-convex subspace of a Heintze group (negatively-curved groups of the form $N\rtimes \R$, where $N$ is a nilpotent Lie group).
\end{enumerate}
\end{theorem}

Theorem \ref{thm:char=min} reduces the classification of hyperbolic structures of abelian-by-cyclic groups based on admissible matrices $\gamma$ to two algebraic computations: computing the invariant subspaces of $\gamma$ and computing the prime factorization of $p$ in the formal power series ring $\z[[x]]$. Computing the invariant subspaces is straightforward using Jordan normal forms (see, e.g., \cite{invariant}), while computing factorizations in formal power series rings can be solved algorithmically (see, e.g., \cite{elliott}).

Theorems \ref{thm:main} and  \ref{thm:char=min} completely recover the main theorems of \cite{AR} and \cite{Qp}. Moreover, they can be applied to describe the poset $\mathcal H(G)$ for numerous new abelian-by-cyclic groups. We give some examples below;. To further substantiate the extent to which the result can be applied, we note that admissible matrices are very common. In particular, the result applies to an abundant family of examples, including when $p$ is monic, irreducible and expanding or has a square-free constant term. Exact descriptions of the poset in these cases may be found in \cite[Corollaries 1.6 and 1.7]{ABR2}.

\begin{example}\label{ex:lamp}
Consider a lamplighter group $(\z/n\z)\wr \z$, where $n\geq 2$. Then $\mathcal H((\z/n\z)\wr \z)$ is as pictured in Figure \ref{fig:lamplighter}: there is a unique lineal structure and two copies of the poset $\operatorname{Sub}(\z/n\z)$ of subgroups of $\z/n\z$ intersecting in the unique lineal structure; all other structures from $\operatorname{Sub}(\z/n\z)$ represent quasi-parabolic action. Every structure can be represented by action on a tree or on a point. This description of $\mathcal H((\z/n\z)\wr \z)$ was first obtained in \cite{Qp}.

\begin{figure}[ht]
\centering
\begin{tikzpicture}[scale=0.6]

\node[circle, draw, minimum size=0.8cm] (triv) at (0,-3) {$*$};
\node[circle, draw, minimum size=0.8cm] (lin) at (0,0) {$\R$};
\node[circle, draw, minimum size=0.8cm] (bs) at (6,3) {$T_+$};
\node[circle, draw, minimum size=0.8cm] (Bs) at (-6,3) {$T_-$};

\node (low1) at (2,1) {};
\node (low2) at (3,1) {};
\node (low3) at (4,1) {};

\node (up1) at (2,1.1) {};
\node (up2) at (3,1.1) {};
\node (up3) at (4,1.1) {};

\node (Low1) at (-2,1) {};
\node (Low2) at (-3,1) {};
\node (Low3) at (-4,1) {};

\node (Up1) at (-2,1.1) {};
\node (Up2) at (-3,1.1) {};
\node (Up3) at (-4,1.1) {};

\draw[thick] (triv) -- (lin);
\draw[thick] (lin) -- (low1);
\draw[thick] (lin) -- (low2);
\draw[thick] (lin) -- (low3);

\draw[thick] (lin) -- (Low1);
\draw[thick] (lin) -- (Low2);
\draw[thick] (lin) -- (Low3);

\draw[thick] (up1) -- (bs);
\draw[thick] (up2) -- (bs);
\draw[thick] (up3) -- (bs);

\draw[thick] (Up1) -- (Bs);
\draw[thick] (Up2) -- (Bs);
\draw[thick] (Up3) -- (Bs);

\draw[thick, dotted, red, rotate=296.5] (-0.05, 3.4) ellipse (60pt and 140pt);

\draw[thick, dotted, blue, rotate=63.5] (-0.05, 3.4) ellipse (60pt and 140pt);

\node[red] (ideals) at (6,-1) {$\operatorname{Sub}(\z/n\z)$};

\node[blue] (others) at (-6,-1) {$\operatorname{Sub}(\z/n\z)$};

\end{tikzpicture}
\caption{$\mathcal{H}((\z/n\z)\wr \z)$}
\label{fig:lamplighter}
\end{figure}
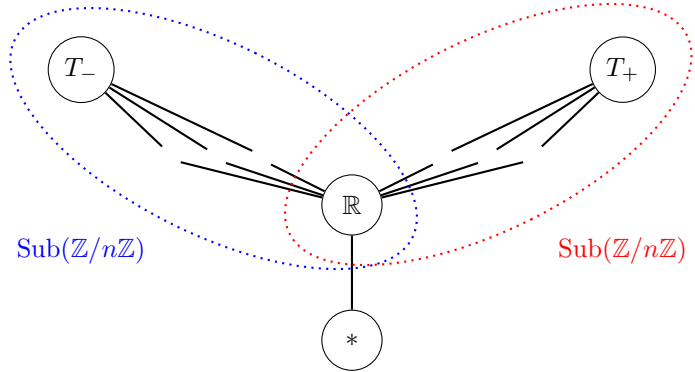
\end{example}

\begin{example}\label{ex:bs}
Consider the solvable Baumslag-Solitar groups $BS(1,n)=\langle t, a : tat^{-1}=a^n\rangle$ for $n\in \z\setminus \{-1,0,1\}$, and let $n=\pm  p_1^{n_1}\cdots p_r^{n_r}$ be the prime factorization of $n$. Then $\mathcal H(BS(1,n))$ is as pictured in Figure \ref{fig:bs1n}: there is a unique lineal structure, a copy of the poset $2^{\{1,\ldots,r\}}$ of subsets of $\{1,\ldots,r\}$ containing the single lineal structure, while all other structures are quasi-parabolic. There is a single additional quasi-parabolic structure represented by the action on the hyperbolic plane $\mathbb{H}^2$. Every element of $2^{\{1,\ldots,r\}}$can be represented by an action on a tree.  This description of $\mathcal H(BS(1,n))$ was first obtained in \cite{AR}.

\begin{figure}[ht]
\centering

\begin{tikzpicture}[scale =0.6]

\node[circle, draw, minimum size=0.8cm] (triv) at (0,-3) {$*$};
\node[circle, draw, minimum size=0.8cm] (lin) at (0,0) {$\R$};
\node[circle, draw, minimum size=0.8cm] (hyp) at (-3,3) {$\mathbb{H}^2$};
\node[circle, draw, minimum size=0.8cm] (bs) at (6,3) {$T$};

\node (low1) at (2,1) {};
\node (low2) at (3,1) {};
\node (low3) at (4,1) {};

\node (up1) at (2,1.1) {};
\node (up2) at (3,1.1) {};
\node (up3) at (4,1.1) {};

\draw[thick] (triv) -- (lin) -- (hyp);
\draw[thick] (lin) -- (low1);
\draw[thick] (lin) -- (low2);
\draw[thick] (lin) -- (low3);

\draw[thick] (up1) -- (bs);
\draw[thick] (up2) -- (bs);
\draw[thick] (up3) -- (bs);

\draw[thick, dotted, red, rotate=296.5] (-0.05, 3.4) ellipse (60pt and 140pt);

\node[red] (poset) at (6,-1) {$2^{\{1,\ldots,r\}}$};

\end{tikzpicture}
\caption{$\mathcal{H}(BS(1,n))$}
\label{fig:bs1n}

\end{figure}
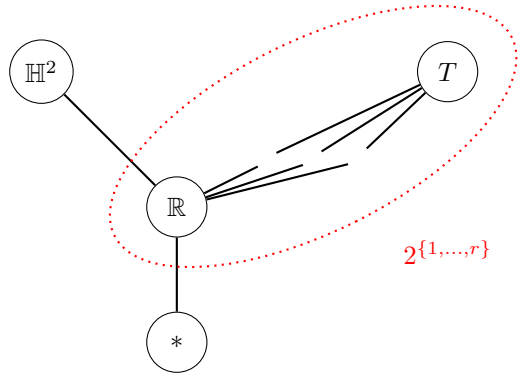
\end{example}

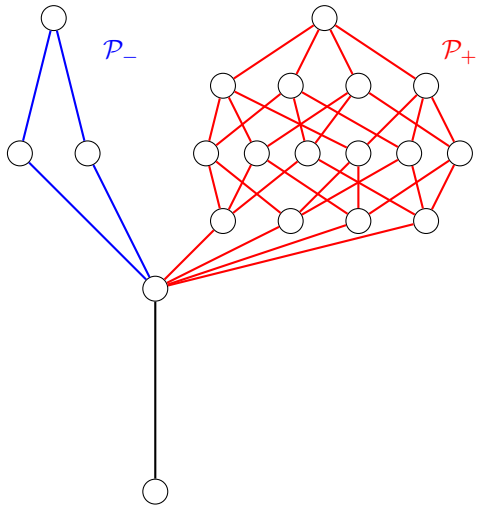
\begin{figure}[ht]
\centering
\begin{tikzpicture}[scale=0.9]

\node[circle, draw, minimum size=0.1cm] (triv) at (0,-3) {};
\node[circle, draw, minimum size=0.1cm] (lin) at (0,0) {};

\node[circle, draw, minimum size=0.1cm] (low1) at (1,1) {};
\node[circle, draw, minimum size=0.1cm] (low2) at (2,1) {};
\node[circle, draw, minimum size=0.1cm] (low3) at (3,1) {};
\node[circle, draw, minimum size=0.1cm] (low4) at (4,1) {};

\node[circle, draw, minimum size=0.1cm] (med12) at (0.75,2) {};
\node[circle, draw, minimum size=0.1cm] (med13) at (1.5,2) {};
\node[circle, draw, minimum size=0.1cm] (med14) at (2.25,2) {};
\node[circle, draw, minimum size=0.1cm] (med23) at (3,2) {};
\node[circle, draw, minimum size=0.1cm] (med24) at (3.75,2) {};
\node[circle, draw, minimum size=0.1cm] (med34) at (4.5,2) {};

\node[circle, draw, minimum size=0.1cm] (up123) at (1,3) {};
\node[circle, draw, minimum size=0.1cm] (up124) at (2,3) {};
\node[circle, draw, minimum size=0.1cm] (up134) at (3,3) {};
\node[circle, draw, minimum size=0.1cm] (up234) at (4,3) {};

\node[circle, draw, minimum size=0.1cm] (up) at (2.5,4) {};

\node[circle, draw, minimum size=0.1cm] (left1) at (-1,2) {};
\node[circle, draw, minimum size=0.1cm] (left2) at (-2,2) {};

\node[circle, draw, minimum size=0.1cm] (left) at (-1.5,4) {};

\node (p+) at (4.5,3.5) [red] {$\mathcal P_+$};

\node (p-) at (-0.5,3.5) [blue] {$\mathcal P_-$};

\draw[thick] (triv) -- (lin);
\draw[thick, red] (lin) -- (low1);
\draw[thick, red] (lin) -- (low2);
\draw[thick, red] (lin) -- (low3);
\draw[thick, red] (lin) -- (low4);
\draw[thick, red] (low1) -- (med12);
\draw[thick, red] (low1) -- (med13);
\draw[thick, red] (low1) -- (med14);
\draw[thick, red] (low2) -- (med12);
\draw[thick, red] (low2) -- (med23);
\draw[thick, red] (low2) -- (med24);
\draw[thick, red] (low3) -- (med13);
\draw[thick, red] (low3) -- (med23);
\draw[thick, red] (low3) -- (med34);
\draw[thick, red] (low4) -- (med14);
\draw[thick, red] (low4) -- (med24);
\draw[thick, red] (low4) -- (med34);
\draw[thick, red] (med12) -- (up123);
\draw[thick, red] (med12) -- (up124);
\draw[thick, red] (med13) -- (up123);
\draw[thick, red] (med13) -- (up134);
\draw[thick, red] (med14) -- (up124);
\draw[thick, red] (med14) -- (up134);
\draw[thick, red] (med23) -- (up123);
\draw[thick, red] (med23) -- (up234);
\draw[thick, red] (med24) -- (up124);
\draw[thick, red] (med24) -- (up234);
\draw[thick, red] (med34) -- (up134);
\draw[thick, red] (med34) -- (up234);
\draw[thick, red] (up123) -- (up);
\draw[thick, red] (up124) -- (up);
\draw[thick, red] (up134) -- (up);
\draw[thick, red] (up234) -- (up);
\draw[thick, blue] (lin) -- (left1);
\draw[thick, blue] (lin) -- (left2);
\draw[thick, blue] (left1) -- (left);
\draw[thick, blue] (left2) -- (left);
\end{tikzpicture}

\caption{Poset $\mathcal H(G(\gamma))$ for the chosen admissible matrix $\gamma$. Every equivalence class to the left of the figure is represented by an action on a Heintze group, while every class to the right is represented by an action on a tree.}
\label{fig:posetexamples}
\end{figure}

\begin{example}[New example]
Consider a group $G(\gamma)=\langle \z^n , t : tzt^{-1}=\gamma z \text{ for all } z\in \z^n\rangle$, where $\gamma\in M_n(\z)$ is  the admissible matrix \[\gamma=\begin{pmatrix} 0 & 0& -210 \\ 1 & 0 & -1\\ 0 & 1 & 0 \end{pmatrix}.\]  Theorem \ref{thm:char=min} can be used to show that $\mathcal H(G(\gamma))$ is as pictured in Figure \ref{fig:posetexamples}.
\end{example}

Part of the description of Theorem \ref{thm:char=min}  holds in greater generality. In particular, we can drop the requirement that $\z^n$ is a cyclic module, while recovering much of the information about the structure of $\HlG$: the general structure of $\HlG$ and parts (2) and (3) of Theorem \ref{thm:char=min} continue to hold. While we lose the exact description of $\mathcal{P}_+(G)$, it continues to remain a lattice (see \cite[Theorem 1.8]{ABR2}).

\subsection{Groups with higher rank abelianizations}

While the machinery of \cite{ABR2} applies to many groups, it is not specific enough to classify the hyperbolic actions of solvable groups whose abelianizations have higher rank. Thus, the techniques described in the section above do not immediately extend  to such groups.  The paper \cite{ABR} began the work necessary to extend the theory in \cite{Amen} to general finitely generated solvable groups. In particular, \cite{ABR} develops a strong definition of confining subsets for semidirect products of the form $H \rtimes \z^n, n \geq 2$, which are (up to equivalence) in on-to-one correspondence with hyperbolic actions of such groups when they are solvable (also more generally, when $H$ is amenable).  

In particular, this allows us to  completely describe $\HlG$  of certain solvable groups that were previously out of reach. In contrast to the examples from the previous sections, whose posets are finite, these groups always admit uncountably many inequivalent actions on lines, because $\z^n$ does when $n \geq 2$ (see \cite[Example~4.23]{ABO}). However, the  remaining hyperbolic structures can be understood for certain solvable groups.

The two main theorems of \cite{ABR}, namely \cite[Theorems 1.1 and 1.2]{ABR}, give a correspondence between \emph{higher rank} confining subsets and quasi-parabolic actions for groups $H\rtimes \z^n$. This should be compared to \cite[Theorem~4.1]{Amen} in the case $n=1$. 

To illustrate the use of this upgraded machinery, \cite{ABR} gives a complete description of $\HlG$ for a class of groups related to solvable Baumslag-Solitar groups.  If $k=p_1^{m_1}\cdots p_n^{m_n}$ is the prime factorization of $k$ with $n \geq 2$, then the \textit{generalized solvable Baumslag-Solitar group} $G_k$ is defined as  
 $G_k := \z\left[\frac1k\right] \rtimes_\gamma \z^n$, where the image under $\gamma$ of the  $i$\textsuperscript{th} generator of $\z^n$ acts on $\z[\frac1k]$ by multiplication by $p_i^{m_i}$. This group has a presentation \[G_k = \left\langle a, t_1,\ldots,t_n \ \big\vert \ [t_i,t_j]=1, t_iat_i^{-1}=a^{p_i^{m_i}} \text{ for all } i,j\right\rangle,\] where $a$ corresponds to a normal generator of the subgroup $\z[\frac{1}{k}]$.  When $k$ is a power of a single prime number, the structure of $\HlG$ can be recovered from the results of \cite{AR} discussed in the previous section.

 \begin{theorem}\label{thm:Z1k}
 For any $k\geq 2$ which is not a power of a prime, the poset $\mathcal H(G_k)$ has the following structure: $\mathcal H_{qp}(G_k)$, the subposet of quasi-parabolic actions, consists of $n+1$ incomparable elements. Each quasi-parabolic action dominates a single lineal action; there are uncountably many lineal actions; and all lineal actions dominate a single elliptic action (see Figure \ref{fig:Z[1/k]Structure}). Moreover, every element of $\mathcal{H}(G_k)$ is represented either by an action on a tree, the hyperbolic plane, or a point.
 \end{theorem} 


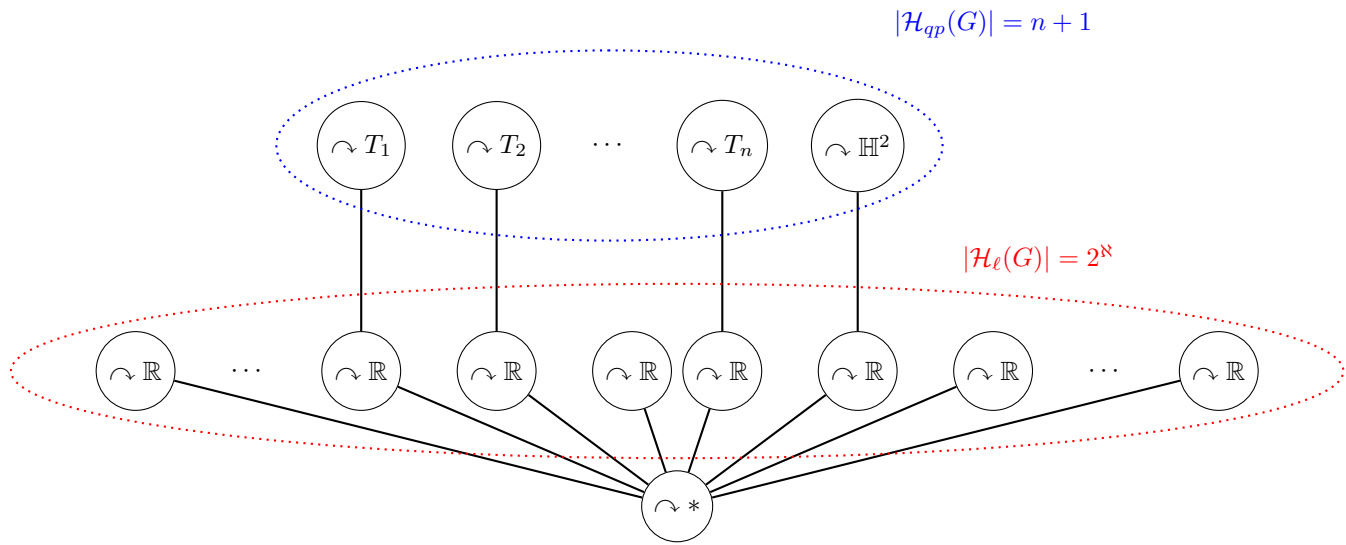
\begin{figure}[ht]
\centering
\begin{tikzpicture}[scale=0.6]

\node[circle, draw, minimum size=0.8cm] (triv) at (0,-3) {$\acts *$};
\node[circle, draw, minimum size=0.8cm] (lin1) at (1,0) {$\acts \R$};
\node[circle, draw, minimum size=0.8cm] (lin2) at (-1,0) {$\acts \R$};
\node[circle, draw, minimum size=0.8cm] (lin3) at (4,0) {$\acts \R$};
\node[circle, draw, minimum size=0.8cm] (lin4) at (-4,0) {$\acts \R$};
\node[circle, draw, minimum size=0.8cm] (lin5) at (7,0) {$\acts \R$};
\node[circle, draw, minimum size=0.8cm] (lin6) at (-7,0) {$\acts \R$};
\node[circle, draw, minimum size=0.8cm] (lin7) at (-12,0) {$\acts \R$};
\node[circle, draw, minimum size=0.8cm] (lin8) at (12,0) {$\acts \R$};

\node[circle, draw, minimum size=0.8cm] (qp1) at (-7,5) {$\acts T_1$};
\node[circle, draw, minimum size=0.8cm] (qp2) at (-4,5) {$\acts T_2$};
\node[circle, draw, minimum size=0.8cm] (qpn) at (1,5) {$\acts T_n$};
\node[circle, draw, minimum size=0.8cm] (qph) at (4,5) {$\acts\mathbb{H}^2$};

\draw[thick] (triv) -- (lin1);
\draw[thick] (triv) -- (lin2);
\draw[thick] (triv) -- (lin3);
\draw[thick] (triv) -- (lin4);
\draw[thick] (triv) -- (lin5);
\draw[thick] (triv) -- (lin6);
\draw[thick] (triv) -- (lin7);
\draw[thick] (triv) -- (lin8);

\draw[thick] (lin1) -- (qpn);
\draw[thick] (lin4) -- (qp2);
\draw[thick] (lin6) -- (qp1);
\draw[thick] (lin3) -- (qph);

\draw[thick, dotted, red, rotate=90] (0, 0) ellipse (55pt and 420pt);
\draw[thick, dotted, blue, rotate=90] (5,1.5) ellipse (60pt and 210pt);

\node[red] (lineals) at (8,2.5) {$|\Hll_\ell(G)| = 2^\aleph$};
\node[blue] (focals) at (7,7.7) {$|\Hll_{qp}(G)| =n+1$};
\node (others1) at (-1.5,5) {$\cdots$};
\node (others2) at (9.5,0) {$\cdots$};
\node (others3) at (-9.5,0) {$\cdots$};

\end{tikzpicture}
\caption{$\mathcal{H}((\z/n\z)\wr \z)$}
\label{fig:lamplighter}
\end{figure}

Theorem \ref{thm:Z1k} reveals that the natural actions of $\operatorname{PSL}_2\left(\z\left[\frac{1}{k}\right]\right)$ on the hyperbolic plane and Bruhat-Tits trees give rise to all of the non-elementary hyperbolic actions of $G_k$. Using a characterization of the BNS invariant of a finitely generated group in terms of its actions on trees due to Brown \cite{brown}, Theorem \ref{thm:Z1k} can be used to compute the BNS invariants of the groups $G_k$, recovering a result of Sgobbi and Wong \cite{SgobbiWong}. This adds to the motivation to study the structure of $\HlG$ as a source of information to compute the BNS invariant of new groups.

\subsection{Amenable groups}

In this subsection, we focus on results pertaining to different types of \emph{amenable groups} (including nilpotent and solvable groups) and when they admit hyperbolic actions. An amenable group has no free subgroups, and so their non-elementary actions are necessarily quasi-parabolic. The work contained in this section comes from \cite{PropNL}, which focused on Property $\nl$. 

The motivation to study property $\nl$ comes from the natural question to study groups that are ``inaccessible" from the viewpoint of hyperbolic actions. i.e. groups that do not admit any interesting hyperbolic actions.  The approach can be compared with the study of Kazhdan's Property (T), which imposes rigidity conditions on the group by forbidding isometric actions on Hilbert spaces without a global fixed point. 

\begin{definition}
A group $G$ satisfies Property $\nl$ -- standing for \emph{No Loxodromics} -- if no action by isometries of $G$ on a hyperbolic space admits a loxodromic element. A group has Property \emph{Hereditary $\nl$} if all of its finite-index subgroups have Property $\nl$. 
\end{definition}

In other words, a group is $\nl$ if its only possible isometric actions on hyperbolic spaces are elliptic or horocyclic. These are the only types of actions we cannot rule out since every group admits elliptic actions (e.g.\ the trivial action on a single point) and every countably infinite group admits horocyclic actions (e.g.\ combinatorial horoballs on Cayley graphs \cite{GrovesManning} or trees in case the group is not finitely generated \cite{Serre}). 

Besides the obvious finite groups, several families of groups were known to have few or no interesting actions on hyperbolic spaces, such as Burnside groups, Tarski monsters, Grigorchuk's group (as they are infinite torsion groups) and Thompsons group $V$ \cite{anthony}. The paper \cite{PropNL} initiates a formal systematic study of this property, which interestingly, has connections to other fixed point properties (see \cite[Section 6.1]{PropNL}). One important connection is to the structure of $\HlG$, which follows from the appendix of \cite{PropNL}.

\begin{prop}\cite[Corollary 6.7]{PropNL} A group $G$ has Property $\nl$ if and only if $\HlG = \emptyset$.
\end{prop}

In this setting, the structure of amenable groups allows us to prove many strong results, starting with the following.

\begin{prop}\label{prop:Amenable}\cite[Proposition 3.1]{PropNL}
An amenable group $G$ has Property $\nl$ (resp. hereditary $\nl$) if and only if every homomorphism from $G$ (resp. a finite-index subgroup of $G$) to $\mathbb{R} \rtimes \mathbb{Z}/2\mathbb{Z}$ has image of order at most $2$. If $G$ is finitely generated, this amounts to saying that it does not surject (resp. virtually surject) onto $\mathbb{Z}$ nor~$\mathbb{D}_\infty$. 
\end{prop}

Abelian and nilpotent groups are widely studied classes of groups among amenable groups. The rigid structure of both these classes of groups allows us to prove even stronger results about them and Property $\nl$.

\begin{prop}\label{prop:nosubsemi}\cite[Proposition 3.2]{PropNL}
Let $G$ be a group with no non-abelian free sub-semigroup. Then $G$ cannot admit a non-elementary action on a hyperbolic space. If $G$ is amenable and admits a lineal action, then the action factors through a homomorphism $G \to \mathbb{R} \rtimes \mathbb{Z}/2\mathbb{Z}$. 
\end{prop}

The above proposition is interesting as it involves the introduction and construction of a \emph{Buseman quasi-cocylce}. The Busemann quasimorphism is classically defined only for actions fixing a point in the boundary; but here we must introduce a generalization to non-orientable lineal actions, where instead two points on the boundary are fixed setwise. Consequently, an analogue of Lemma \ref{lem:constolineal} is also developedfor non-orientable lineal actions , see \cite[Corollary 2.11 and Lemma 2.14]{PropNL} for details. 

As a direct consequence of the above result, abelian groups, nilpotent groups, groups of sub-exponential growth, virtually nilpotent groups and supramenable groups satisfy $\HlG = \Hll_{e}(G) \sqcup \Hll_\ell(G)$. Moreover, every lineal action of an abelian or nilpotent group is necessarily oriented, and factors through a homomorphism $G \to \R$. This leads to the following results. 

\begin{cor}\cite[Corollary 3.5 and 3.6]{PropNL}
An abelian group $G$ has property $\nl$ (or hereditary $\nl$) if and only if it is a torsion group. A nilpotent group $G$ has property $\nl$ if and only if every homomorphism to $\mathbb{R}$ is trivial.
\end{cor} 

Lastly, we discuss property $\nl$ among solvable groups. Unlike abelian and nilpotent groups which have no non-elementary actions, these can admit (many) focal actions, as seen in Section \ref{sec:solvgrps}. The following result clarifies which solvable groups admit focal actions. 

\begin{prop}\cite[Proposition 3.7]{PropNL}
A finitely generated solvable group is either virtually nilpotent or contains a finite-index subgroup admitting a focal action. 
\end{prop}

\begin{cor}\cite[Corollary 3.8]{PropNL}
A finitely generated solvable group does not admit non-elementary actions if and only if it is virtually nilpotent. 
\end{cor} 

A (virtually) nilpotent group $G$ is (virtually) cyclic if and only if its abelianization is cyclic, in which case $\HlG =\{ [G], [X]\}$, where $X$ is a finite generating set for $G$. If the abelianization of a nilpotent group $G$ has higher rank, then the group admits uncountably many pairwise inequivalent actions on lines (which can be classified by inequivalent homomorphisms $G \to \R$) in addition to the trivial structure. This completely classifies the poset structure for virtually nilpotent groups.


\section{Thompson like groups}\label{sec:thompson}

While the work contained in \cite{PropNL} additionally covers the stability of property $\nl$ under group operations and its connections to other fixed point properties, one of its main results produces many new interesting examples of groups with property $\nl$. The main tool used to produce new examples is a dynamical criterion for groups with ``rich" actions by \emph{homeomorphisms} on compact Hausdorff spaces. Interestingly, the actions by homeomorphims record important information about the group and its subgroups, which in turn restricts the isometric hyperbolic actions such a group can admit. In particular, the criterion applies to many groups acting on the circle or a Cantor set. The result draws inspiration from the strategy used in \cite{anthony} for Thompsons group $V$, but applies to a much larger class of ``Thompson-like" groups. 

\begin{theorem}\cite[Theorem 5.1]{PropNL}Let $G$ be a group acting faithfully on a compact Hausdorff space $X$. Suppose that there is a basis $\mathcal{I}$ of non-dense and non-empty open sets in $X$, such that the following holds:
\begin{itemize}
	\item[1.] \emph{($\mathbf{C}$: Complements)} For every $I \in \mathcal{I}$ there exists $J \in \mathcal{I}$ such that $I^c \subset J$.
    \item[2.] \emph{($2 \mathbf{T}$: Double transitivity)} $G$ preserves $\mathcal{I}$ and acts \emph{doubly transitively} on it: for every two pairs $(I, J), (I', J') \in \mathcal{I}^{(2)}$, there exists $g \in G$ such that $(gI, gJ) = (I', J')$.
    \item[3.] \emph{($3 \mathbf{T}$: Weak triple transitivity)} For every $g, h \in G$ there exist $(M, N, P) \in \mathcal{I}^{(3)}$ such that $(M, N, P)$ and $(g M, h N, P)$ are in the same $G$-orbit in $\mathcal{I}^{(3)}$.
    \item[4.] \emph{($\mathbf{L}$: Local action)} Let $(I, J, K) \in \mathcal{I}^{(3)}$, and let $g, h \in G$ be such that $g I = I$ and $h J = J$. Then there exists $b \in G$ such that $b|_I = g|_I$ and $b|_J = h|_J$ and $b|_K = id|_K$.
\end{itemize}
Then $G$ has no general type hyperbolic actions.
\end{theorem}

Combined with some well known properties about the existence (or lack thereof) of unbounded quasimorphisms, \cite{PropNL} produces the following list of examples of groups with property $\nl$. 

\begin{theorem}\cite[Theorem 1.6]{PropNL}
The following groups are hereditary $\nl$:
\begin{enumerate}[noitemsep]
    \item Thompsons group $T$
    \item Higman--Thompson groups $V_n(r)$. As a special case, this recovers Thompsons group $V = V_2(1)$
    \item Some R\"{o}ver--Nekrashevych groups, including Neretin's group
    \item Twisted Brin--Thompson groups $SV_\Gamma$ whenever $S$ is a countable faithful $\Gamma$-set
    \item Higman--Stein--Thompson groups $T_{n_1, \ldots, n_k}$, with $k \geq 1$ and $n_1 = 2, n_2 = 3$
    \item The golden ratio Thompson group $T_\tau$
    \item The finitely presented group $S$ of piecewise projective homeomorphisms of the circle, from \cite{yashS}
    \item Symmetrizations of Higman--Thompson groups $QV_n(r)$ and $QT$ 
\end{enumerate}
\end{theorem}

In particular, the result on twisted Brin--Thompson groups shows that groups with property $\nl$ are ubiquitous, in the following sense.

\begin{cor}\cite[Corollary 1.7]{PropNL}\label{cor:surpriseqi}
Every finitely generated group quasi-isometrically embeds into a finitely generated simple group, which is hereditary $\nl$.
\end{cor}

We warn the reader that our methods ensure Property $\nl$ under the existence of flexible enough actions on compact Hausdroff spaces, and in no way claims that all ``Thompson-like'' groups have Property $\nl$. Indeed, Braided Thompson groups have rich actions on hyperbolic spaces \cite{bV}, and serve as a counterexample to this notion.

\section{Higher rank lattices}\label{sec:lattices}

Recently, work has emerged pertaining to the hyperbolic actions of irreducible lattices in a semi-simple group $G$ of rank $\geq 2$, which we cover in this section. Note that it follows from \cite{haettel} that if all simple factors
of the ambient product have rank $\geq 2$, then the lattice has property $\nl$. The case when $G$ has simple factors of rank $1$ is covered in \cite{bcfs}. This is an interesting case to consider as rank $1$ simple groups have a natural geometric action on a hyperbolic space, so the lattices in $G$ admit non-trivial actions on hyperbolic spaces via their projections on the rank $1$ simple factors.

The setup considered in \cite{bcfs} is slightly different from that of cobounded hyperbolic actions. Given an action on a hyperbolic space, one can make a larger hyperbolic space containing the first one as a quasiconvex subspace while maintaining important properties of the original action. For instance, this can be achieved by attaching geodesic rays equivariantly. To take this possibility into account, the work in \cite{bcfs} focuses on \emph{coarsely minimal} actions.

\begin{definition} We say that an action $G \acts X$ on a hyperbolic space is coarsely
minimal if $X$ is unbounded, the limit set of $G$ on $\partial X$ is not a single point, and every
quasi-convex $G$-invariant subset of $X$ is coarsely dense.
\end{definition}

Although coarsely minimal actions provide a broader setup than cobounded actions, it closely links to the structure of $\HlG$. Indeed, for any hyperbolic structure, the action is either coarsely minimal or elliptic. Thus, the number of hyperbolic structures is one more than the number of equivalence classes of coarsely minimal actions; where the equivalence on coarsely minimal actions is provided roughly by $G-$equivariant quasi-isometries; we refer the reader to \cite[Definition 3.5]{bcfs} for a precise description. Consequently, if one classifies all coarsely minimal actions up to equivalence, the elements of $\HlG$ can be easily deduced. 

In what follows, the notion of a standard rank one group is taken from \cite[Theorem D]{Amen} -- a locally compact group $G$ is a standard rank one group if it has no nontrivial compact
normal subgroups and either

\begin{enumerate}
\item $G$ is the group of isometries or orientation-preserving isometries of a rank
one symmetric space $X$ of non-compact type, or
\item $G$ has a continuous, proper, faithful action by automorphisms on a locally finite non-elementary tree $T$ , without inversions and with exactly two orbits of vertices, such that the action of $G$ on the set of ends $\partial T$ is 2-transitive.
\end{enumerate}

The symmetric space $X$ in case (1) and the tree $T$ in case (2) are called the model
spaces for the standard rank one group. While standard rank one groups of type (1) correspond to real Lie groups of rank one, type (2) includes simple algebraic groups over non-archimedean local fields of rank one. With this terminology, the main result of \cite{bcfs} is as follows.

\begin{theorem}\cite[Theorem 1.1]{bcfs} Let $N \geq n \geq 0$ be integers. Let $G = \displaystyle \prod_{i=1} ^N G_i$ be a product of $N$ locally compact groups, where  $G_i$ is a standard rank one
group for all $i \in \{1,\cdots, n\}$, and $G_j$ is a simple algebraic group defined over a
local field $k_j$ with $\op{rk}_{k_j}
(G_j ) \geq 2$ for all $j \in \{n + 1,\cdots, N\}$. Let $\G < G$ be a lattice. Assume that $n \geq 2$ or that $N > n$. If $N > 1$, assume in addition that $\G$ has a dense projection to each proper sub-product.
Then any coarsely minimal action of $\G$ on a geodesic hyperbolic space is equivalent to one of the actions
$$\G \to G \xrightarrow[]{pr_i} G_i \to Isom(X_i, d_i) \hspace{10pt} (\text{for }1 \leq i \leq n)$$
where each $X_i$ is a rank-one symmetric space or a tree, corresponding to the standard rank one factor $G_i$ being of type (1) or (2).
\end{theorem}

In other words, the only possible actions on hyperbolic spaces of such lattices arise from their projections to the rank 1 simple factors. The result allows for the classification of the poset $\HlG$ for new groups, including the famous Burger-Mozes groups, which are irreducible lattices $\Ga$ in a product $G_1 \times G_2$ of two standard rank one groups $G_i < Aut(T_i), i =1,2.$ \cite{burgermozes}.

\begin{cor}\cite[Example 1.5]{bcfs} Every Burger–Mozes group has two non-trivial hyperbolic structures. Specifically, the only non-trivial hyperbolic structures of $\G$ come from the actions on the trees $T_1$ and $T_2$.  
\end{cor} 

Other new examples from \cite{bcfs} include producing groups with exactly $n$ hyperbolic structures, for each $n \geq 1$; different from the ones previously produced in \cite{ABO}.

\section{New examples and open questions} 

In this last section, we record some examples of interest and pose related open problems that could inspire future directions of research. 

\paragraph{Groups with finitely many quasi-parabolic actions.}

By Example \ref{ex:lamp}, the lamplighter group $\mathcal{L}_2 = (\z/2\z) \wr \z$ has exactly 2 quasi-parabolic structures. It follows from \cite[Lemma 4.20]{ABO}, that the group $\displaystyle H_i = \bigoplus_{j =1} ^i \mathcal{L}_2$ has exactly $2i$ quasi-parabolic structures for all $i \geq 1$. 

Further, the group $G_6$ from Theorem \ref{thm:Z1k} has exactly $3$ quasi-parabolic structures. Consequently, the groups $K_i = H_i \oplus G_6$ have exactly $2i +3$ quasi-parabolic structures for all $i \geq 1$. This produces groups with exactly $n$ quasi-parabolic structures for every $n \geq 2$, almost completely answering an open problem from \cite{ABO}. It remains to consider the following. 

\begin{ques} Does there exist a group $G$ such that $|\Hll_{qp}(G)| =1$? 
\end{ques}

Such an example, if it exists, could lead to an interesting study of inherent properties that a group must have in order for it to admit only one-quasi-parabolic action. If such a group cannot exist, then the proof of this fact would also be interesting.

\paragraph{Embeddings into groups with property $\nl$.}
Recall that Corollary \ref{cor:surpriseqi} allows us to quasi-isometrically embed any countable group into a finitely generated, simple group with property $\nl$. While this result uses twisted Brin-Thompson groups, we give a new construction that allows us to \emph{embed} any countable group into a simple group with property $\nl$. For this, we need the following result by Osin. 

\begin{theorem}\cite[Theorem 1.1]{osinembed}\label{thm:osinembed} Any countable group $G$ can be embedded into a 2–generated group $C$ such that any two elements of the same order are conjugate in $C$.
\end{theorem}

The above immediately allows us to prove the following. 

\begin{prop} Any countable group $G$ can be embedded into a finitely generated group with property $\nl$. 
\end{prop}
\begin{proof} Given $G$, we use Theorem \ref{thm:osinembed} to embed $G$ into the group $C$. We claim that $C$ has property $\nl$. Indeed, any element $c$ of infinite order is conjugate to $c^2$ in $C$ (since they are both infinite order). It follows that the group $\langle c \rangle$ cannot be quasi-isometrically embedded in any action on a hyperbolic space; which implies that $c$ cannot be a loxodromic element. 
\end{proof}

\begin{remark} Observe that the above proof can be adapted to see that any group with finitely many conjugacy classes also has property $\nl$. In particular, the torsion free finitely generated group with exactly two conjugacy classes constructed in \cite{osinembed} has property $\nl$.
\end{remark}

We can further utilize the result by Osin to prove the claimed stronger result. 

\begin{prop} Every countable group $G$ with an infinite order element can be embedded into a finitely generated, simple group with property $\nl$.    
\end{prop}
\begin{proof} Given a finitely generated group $G$, consider the countable group $H= \bigoplus_n (\z/n\z) \oplus G$. We first embed $H$ into a countable, simple group $S$. Next, we embed $S$ into the group $C$ from Theorem \ref{thm:osinembed}. Then $G$ embeds into $C$ as well. Clearly, $C$ is finitely generated and has property $\nl$; this follows from the proof of the previous proposition. So, it remains to show that $C$ is simple in this case.

To this end, consider any element $a \in C$. Then $a^C \cap S \neq \{1\}$, as $S \leq C$ contains elements of all orders. This means that all elements of $C$ are conjugate to an element of $S$. Suppose that $C$ contains a non-trivial normal subgroup $N$, we consider the quotient $C/N$. If $a$ is the kernel of this quotient, then so is an element of $S$. As $S$ is simple, this means that $S$ is in the kernel of the quotient, but then so is the entire group $C$. Thus $N= C$, and so $C$ is simple. 
\end{proof}

These different embedding results prompt the following questions.

\begin{ques} Given a finitely generated group $G$, in how many distinct ways can $G$ (quasi-isometrically) embed into a group with property $\nl$?
\end{ques}

\begin{ques} For a countable group $G$ with infinitely many conjugacy classes, what ``algebraic" conditions are sufficient to imply property $\nl$?
\end{ques}

\paragraph{Anosov mapping torii.} An interesting result obtained in \cite{Largest} was the classification of $\HlG$ for the fundamental group of the Anosov mapping torus $\z^2 \rtimes_\phi \z$, where $\phi$ is a matrix in $SL(2, \z)$. The machinery of confining subsets was also used to achieve this result, which were classified in terms of the neighborhoods of the attracting and repelling eigenlines of $\phi$ in $\R^2$. 

\begin{figure}[ht]
\centering

\begin{tikzpicture}[scale=0.5]
\node[circle, draw, minimum size=0.8cm] (triv) at (0,-3) {$\acts *$};
\node[circle, draw, minimum size=0.8cm] (lin) at (0,0) {$\acts \R$};
\node[circle, draw, minimum size=0.8cm] (bs) at (3,3) {$\acts \mathbb{H}^2$};
\node[circle, draw, minimum size=0.8cm] (bs2) at (-3,3) {$\acts \mathbb{H}^2$};

\draw[thick] (triv) -- (lin);
\draw[thick] (lin) -- (bs);
\draw[thick] (lin) -- (bs2);

\end{tikzpicture}
\caption{$\mathcal{H}(G)$}
\label{fig:amt}
\end{figure}
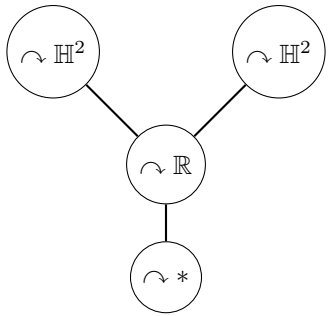

The work in \cite{ABR2} does not currently subsume this result, as it deals with the case when $\phi$ is expanding, which a matrix from $SL(2,\z)$ fails to satisfy. However, it is reasonable, especially in light of part (2) of Theorem \ref{thm:char=min}, if the condition that requires an expanding matrix can be dropped. To be precise, we believe that the following question can be answered affirmatively. 

\begin{ques} Consider the group $G = \z^n \rtimes_\phi \z$, where $\phi \in SL(n, \z)$. Are the quasi-parabolic structures of $G$ are in one-to-one correspondence with the proper invariant subspaces of $\phi$ in $\R^n$? $\HlG$ additionally contains a unique lineal and elliptic structure. 
\end{ques}

\paragraph{Further classification results.}

We expect that the theory developed in \cite{ABR2, ABR} will lay the groundwork for future, more general classifications of hyperbolic actions of metabelian groups. In particular, we may consider the following question. 

\begin{ques} Can Theorems \ref{thm:main} and the techniques developed in \cite{Qp,AR,Largest} be applied to classify the hyperbolic actions of wreath products $A\wr B$ and extensions $A\rtimes B$ when $A$ and $B$ are finitely generated abelian groups?
\end{ques}

As the work contained in \cite{ABR, ABR2} heavily relies on the semi-direct product structure of the group, we may also ask the following. 

\begin{ques} Can the theory of confining subsets be extended to solvable groups that do not necessarily admit a semi-direct product decomposition?
\end{ques}


\bibliographystyle{amsplain}
\bibliography{bibliography}

\providecommand{\bysame}{\leavevmode\hbox to3em{\hrulefill}\thinspace}
\providecommand{\MR}{\relax\ifhmode\unskip\space\fi MR }
\providecommand{\MRhref}[2]{%
  \href{http://www.ams.org/mathscinet-getitem?mr=#1}{#2}
}
\providecommand{\href}[2]{#2}
\begin{thebibliography}{10}

\bibitem{Largest}
C.~R. Abbott and A.~J. Rasmussen, \emph{Largest hyperbolic actions and quasi-parabolic actions in groups}, J. Topol. Anal. (To appear).

\bibitem{ABR2}
Carolyn Abbott, Sahana~H. Balasubramanya, Alexander J.Rasmussen, and Sam Payne, \emph{Valuations, completions, and hyperbolic actions of metabelian groups}, arXiv:2207.12945 (2022).

\bibitem{ABD}
Carolyn Abbott, Jason Behrstock, Daniel Berlyne, Matthew~Gentry Durham, and Jacob Russell, \emph{Largest acylindrical actions and stability in hierarchically hyperbolic groups}, Trans. Amer. Math. Soc. \textbf{Ser. B.} (2021), no.~8, 66--104.

\bibitem{ABO}
Carolyn~R. Abbott, Sahana Balasubramanya, and Denis Osin, \emph{Hyperbolic structures on groups}, Algebr. Geom. Topol. \textbf{19} (2019), no.~4, 1747--1835. \MR{3995018}

\bibitem{ABR}
Carolyn~R. Abbott, Sahana H~Balasubramanya, and Alexander~J. Rasmussen, \emph{Higher rank confining subsets and hyperbolic actions of solvable groups}, Adv. Math. \textbf{424} (2023), Paper No. 109045. \MR{4585997}

\bibitem{AR}
Carolyn~R. Abbott and Alexander~J. Rasmussen, \emph{Actions of solvable {B}aumslag--{S}olitar groups on hyperbolic metric spaces}, Algebr. Geom. Topol. \textbf{23} (2023), no.~4, 1641--1692. \MR{4602410}

\bibitem{bcfs}
U.~Bader, P.-E. Caprace, A.~Furman, and A.~Sisto, \emph{Hyperbolic actions of higher-rank lattices come from rank-one factors}, arXiv:2206.06431, 2022.

\bibitem{Qp}
Sahana~H. Balasubramanya, \emph{Hyperbolic structures on wreath products}, J. Group Theory \textbf{23} (2020), no.~2, 357--383. \MR{4069979}

\bibitem{PropNL}
Sahana~H. Balasubramanya, Francesco Fournier, Anthony Genevois, and Alessandro Sisto, \emph{Property $\mathrm{NL}$ for group actions on hyperbolic spaces}, arXiv:2212.14292 (2022).

\bibitem{bestvinasurvey}
M.~Bestvina, \emph{Groups acting on hyperbolic spaces--a survey}, arXiv preprint arXiv:2206.12916 (2022).

\bibitem{bf}
Mladen Bestvina and Koji Fujiwara, \emph{Bounded cohomology of subgroups of mapping class groups}, Geom. Topol. \textbf{6} (2002), 69--89.

\bibitem{finpres}
R.~Bieri and R.~Strebel, \emph{Almost finitely presented soluble groups}, Comment. Math. Helv. \textbf{53} (1978).

\bibitem{Bowditchacyl}
Brian~H. Bowditch, \emph{Tight geodesics in the curve complex}, Invent. Math. \textbf{171} (2008), no.~2, 281--300. \MR{2367021}

\bibitem{geometry}
M.~R. Bridson, \emph{Invitations to geometry and topology}, Oxford Graduate Texts in Mathematics, vol.~7, ch.~The geometry of the word problem, pp.~29--91, Oxford University Press, 2002.

\bibitem{brown}
K.~S. Brown, \emph{Trees, valuations, and the {Bieri-Neumann-Strebel} invariant}, Invent. Math. \textbf{90} (1987), 479--504.

\bibitem{dunwoody}
Abbott C., \emph{Not all finitely generated groups have universal acylindrical actions}, Proc. Amer. Math. Soc. \textbf{144} (2016), no.~10, 4151--4155.

\bibitem{Amen}
P.-E. Caprace, Y.~Cornulier, N.~Monod, and R.~Tessera, \emph{Amenable hyperbolic groups}, J. Eur. Math. Soc. \textbf{17} (2015), no.~11, 2903--2947. \MR{3420526}

\bibitem{crokekleiner}
C.~B. Croke and B.~Kleiner, \emph{The geodesic flow of a nonpositively curved graph manifold}, Geom. Funct. Anal. \textbf{12} (2002), no.~3, 479--545. \MR{1924370}

\bibitem{dgo}
F.~Dahmani, V.~Guirardel, and D.~Osin, \emph{Hyperbolically embedded subgroups and rotating families in groups acting on hyperbolic spaces}, Mem. Amer. Math. Soc. \textbf{245} (2017), no.~1156.

\bibitem{elliott}
J.~Elliott, \emph{Factoring formal power series over principal ideal domains}, Trans. Amer. Math. Soc. \textbf{366} (2014), no.~8, 3997--4019.

\bibitem{bV}
F.~Fournier-Facio, Y.~Lodha, and M.~C.~B. Zaremsky, \emph{Braided {T}hompson groups with and without quasimorphisms}, Algebr. Geom. Topol. (To appear. \textit{arXiv:2204.05272}, 2022).

\bibitem{anthony}
A.~Genevois, \emph{Hyperbolic and cubical rigidities of {T}hompson's group {$V$}}, J. Group Theory \textbf{22} (2019), no.~2, 313--345. \MR{3918481}

\bibitem{Gen}
Anthony Genevois, \emph{Hyperbolicities in cat(0) cube complexes}, 2019.

\bibitem{invariant}
I.~Gohberg, P.~Lancaster, and L.~Rodman, \emph{Invariant subspaces of matrices with applications}, Society for Industrial and Applied Mathematics, 2006.

\bibitem{Gromov}
M.~Gromov, \emph{Hyperbolic groups}, Essays in group theory, Math. Sci. Res. Inst. Publ., vol.~8, Springer, New York, 1987, pp.~75--263. \MR{919829}

\bibitem{GrovesManning}
Daniel Groves and Jason~Fox Manning, \emph{Dehn filling in relatively hyperbolic groups}, Israel J. Math. \textbf{168} (2008), 317--429. \MR{2448064}

\bibitem{haettel}
T.~Haettel, \emph{Hyperbolic rigidity of higher rank lattices}, Ann. Sci. \'{E}c. Norm. Sup\'{e}r. (4) \textbf{53} (2020), no.~2, 439--468, With an appendix by V. Guirardel and C. Horbez. \MR{4094562}

\bibitem{HRSS}
Mark Hagen, Jacob Russell, Alessandro Sisto, and Davide Spriano, \emph{Equivariant hierarchically hyperbolic structures for 3-manifold groups via quasimorphisms}, 2023.

\bibitem{HagenSusse}
Mark~F. Hagen and Tim Susse, \emph{On hierarchical hyperbolicity of cubical groups}, Israel Journal of Mathematics \textbf{236} (2020), no.~1, 45--89.

\bibitem{hoffman_kunze}
Kenneth Hoffman and Ray Kunze, \emph{Linear algebra}, second ed., Prentice-Hall, Inc., Englewood Cliffs, N.J., 1971. \MR{0276251}

\bibitem{yashS}
Y.~Lodha, \emph{A finitely presented infinite simple group of homeomorphisms of the circle}, J. Lond. Math. Soc. (2) \textbf{100} (2019), no.~3, 1034--1064. \MR{4048731}

\bibitem{Man}
J.~F. Manning, \emph{Actions of certain arithmetic groups on {G}romov hyperbolic spaces}, Algebr. Geom. Topol. \textbf{8} (2008), no.~3, 1371--1402. \MR{2443247}

\bibitem{burgermozes}
M.Burger and S.Mozes, \emph{Finitely presented simple groupa and products of trees}, , C. R. Acad. Sci. Paris S´er. I Math. \textbf{324 no. 7} (1997), 747–752. \MR{1446574}

\bibitem{Osin}
D.~Osin, \emph{Acylindrically hyperbolic groups}, Trans. Amer. Math. Soc. \textbf{368} (2016), no.~2, 851--888. \MR{3430352}

\bibitem{osinembed}
Denis Osin, \emph{Small cancellations over relatively hyperbolic groups and embedding theorems}, Ann. of Math. (2) \textbf{172} (2010), no.~1, 1--39. \MR{2680416}

\bibitem{Sela}
Z.~Sela, \emph{Acylindrical accessibility for groups}, Invent. Math. \textbf{129} (1997), no.~3, 527--565. \MR{1465334}

\bibitem{Serre}
J.-P. Serre, \emph{Arbres, amalgames, {${\rm SL}_{2}$}}, Soci\'{e}t\'{e} Math\'{e}matique de France, Paris, 1977, Avec un sommaire anglais, R\'{e}dig\'{e} avec la collaboration de Hyman Bass, Ast\'{e}risque, No. 46. \MR{0476875}

\bibitem{SSV}
Wagner Sgobbi, Dalton~C. Silva, and Daniel Vendr\'{u}sculo, \emph{The {$R_\infty$} property for nilpotent quotients of generalized solvable {Baumslag-Solitar} groups}, arXiv:2208.02647 (2022).

\bibitem{SgobbiWong}
Wagner Sgobbi and Peter Wong, \emph{The {BNS} invariants of the generalized solvable {Baumslag-Solitar} groups and of their finite index subgroups}, arXiv:2110.14834 (2021).

\bibitem{MV}
Motiejus Valiunas, \emph{Acylindrical hyperbolicity of groups acting on quasi-median graphs and equations in graph products}, Groups Geom. Dyn (2021), no.~15, 143--195.

\end{thebibliography}

\vfill

\hrulefill

\footnotesize
S.H.Balasubramanya; 
Department of Mathematics, Lafayette College, USA; Email: hassanba@lafayette.edu

\end{document}